\documentclass[12pt, reqno]{amsart}
\usepackage{amsmath, amsthm, amscd, amsfonts, amssymb, graphicx, color}
\usepackage[bookmarksnumbered, colorlinks, plainpages]{hyperref}
\hypersetup{colorlinks=true,linkcolor=red, anchorcolor=green, citecolor=cyan, urlcolor=red, filecolor=magenta, pdftoolbar=true}
\input{mathrsfs.sty}
\newtheorem{theorem}{Theorem}[section]
\newtheorem{lemma}[theorem]{Lemma}
\newtheorem{proposition}[theorem]{Proposition}
\newtheorem{corollary}[theorem]{Corollary}
\theoremstyle{definition}

\theoremstyle{remark}
\newtheorem{remark}[theorem]{Remark}
\numberwithin{equation}{section}

\begin{document}
\title[Positivity of $2\times 2$ block matrices of operators]{Positivity of $2\times 2$ block matrices of operators}
\author[M. S. Moslehian, M. Kian, Q. Xu ]{Mohammad Sal Moslehian$^1$, Mohsen Kian$^2$ \MakeLowercase{and} Qingxiang Xu$^3$}
\address{$^1$Department of Pure Mathematics, Center Of Excellence in Analysis on Algebraic Structures (CEAAS), Ferdowsi University of Mashhad, P. O. Box 1159, Mashhad 91775, Iran}
\email{moslehian@um.ac.ir; moslehian@yahoo.com}

\address{$^2$Department of Mathematics, University of Bojnord, P. O. Box 1339, Bojnord 94531, Iran}
\email{kian@ub.ac.ir; kian\_tak@yahoo.com}

\address{$^3$Department of Mathematics, Shanghai Normal University, Shanghai 200234, PR China.}
\email{qxxu@shnu.edu.cn; qingxiang\_xu@126.com}

\dedicatory{Dedicated to the memory of Ronald. G. Douglas (1938-2018)}



\begin{abstract}
We review some significant generalizations and applications of the celebrated Douglas theorem on equivalence of factorization, range inclusion, and majorization of operators. We then apply it to find a characterization of the positivity of $2\times 2$ block matrices of operators on Hilbert spaces and finally describe the nature of such block matrices and provide several ways for showing their positivity.
\end{abstract}
\maketitle

\section{Introduction}
Let $\mathbb{B}(\mathscr{H}, \mathscr{K})$ denote the space of all bounded linear operators from a complex Hilbert space $(\mathscr{H}, \langle\cdot,\cdot\rangle)$ into a Hilbert space $\mathscr{K}$. We stand $\mathbb{B}(\mathscr{H})$ for $\mathbb{B}(\mathscr{H}, \mathscr{K})$ with $\mathscr{H}=\mathscr{K}$ as a $C^*$-algebra with the operator norm $\|\cdot\|$ and the unit $I$. If $\mathscr{H}=\mathbb{C}^n$, we identify $\mathbb{B}(\mathbb{C}^n)$ with the matrix algebra $\mathbb{M}_n(\mathbb{C})$ of $n\times n$ complex matrices. An operator $A$ is called positive if $\langle Ax,x\rangle\geq0$ for all $x\in{\mathcal H }$, and we then write $A\geq 0$. If $A$ is invertible and positive, then we write $A>0$ and say $A$ is strictly positive. For self-adjoint operators $A,B \in \mathbb{B}(\mathscr{H})$, we say $A\leq B~(\mbox{resp.,} A<B)$ if $B-A\geq0~(\mbox{resp.,} B-A>0)$. The sets of positive and strictly positive operators are denoted by ${\mathbb B}_+({\mathscr H})$ and ${\mathbb B}_{++}({\mathscr H})$, respectively. By a $C^*$-algebra we mean a closed $*$-subalgebra of $\mathbb{B}(\mathscr{H})$. A von Neumann algebra is a $C^*$-algebra that admits an isometric surjective embedding onto the dual of a Banach space.

Block matrices of operators are significant, since they allow us to represent an operator by some simpler operators. Sometimes some properties of an operator $T \in \mathbb{B}(\mathscr{H}_1,\mathscr{H}_2)$ could be illustrated by those of some operators of $\mathbb{B}(\mathscr{H}_1\oplus\mathscr{H}_2)$ in a natural way.
In general, if $\mathscr{H}$ is a Hilbert space of arbitrary dimension, we may decompose $\mathscr{H}$ as the direct sum $\mathscr{H}=\mathscr{H}_1\oplus \mathscr{H}_2$ of an appropriate closed subspace $\mathscr{H}_1$, that is the range of a (orthogonal) projection $P$, and its orthogonal complement $\mathscr{H}_2=\mathscr{H}_1^\perp=\mathcal{R} (I-P)$. Any element $A\in \mathbb{B}(\mathscr{H}_1 \oplus \mathscr{H}_2)$ can be expresses as a block matrix $\left(\begin{array}{cc} A_{11} & A_{12} \\ A_{21} & A_{22}\end{array}\right)$, where $A_{11}=PAP, A_{12}=PA(I-P), A_{21}=(I-P)AP$, and $A_{22}=(I-P)A(I-P)$. Evidently, the spaces $\mathscr{H}_1$ and $\mathscr{H}_2$ are reducing under $A$ if and only if $A_{12}=0=A_{21}$. In this case, we say that $A$ is the direct sum of $A_{11}$ and $A_{22}$ and write $A=A_{11}\oplus A_{22}$.

For any operator $A$ between linear spaces, the range and the null space of $A$ are denoted by ${\mathcal R}(A)$ and ${\mathcal N}(A)$, respectively.

In 1966, R. G. Douglas established the following celebrated assertion, known as the Douglas theorem or Douglas majorization theorem.\\
\textbf{Douglas Theorem~}\cite[Theorem 1]{DOU}. {\it If $A, B \in \mathbb{B}(\mathscr{H})$, then the following statements are equivalent:
\begin{enumerate}
\item[{\rm (i)}] $\mathcal{R}(A) \subseteq \mathcal{R}(B)$;
\item[{\rm (ii)}] $A=BC$ for some $C \in \mathbb{B}(\mathscr{H})$;
\item[{\rm (iii)}]$AA^* \leq k^2 BB^*$ for some $k\geq 0$ (or equivalently $\|A^*x\| \leq k\|B^*x\|$ for all $x\in \mathscr{H}$).
\end{enumerate}
Moreover, if {\rm (i)}, {\rm (ii)}, and {\rm (iii)} are valid, then there exists a unique operator $C$ (known as the Douglas Solution in the literature) so that
\begin{enumerate}
\item[{\rm (a)}] $\|C\|^2=\inf\{\mu|AA^*\leq \mu BB^*\}$;
\item[{\rm (b)}] $\mathcal{N}(A) = \mathcal{N}(C)$;
\item[{\rm (c)}] $\mathcal{R}(C) \subseteq \mathcal{R}(B^{*})^{-}$.
\end{enumerate}
}
Let us present a scketch of proof for the equivalences of (i), (ii), and (iii) for the convenience of the readers.
\begin{proof}
(i) $\Rightarrow$ (ii). For $x\in \mathscr{H}$ there is $z\in \mathscr{H}$ such that $Ax=Bz$. Since $\mathscr{H}={\mathcal N}(B)\oplus {\mathcal N}(B)^{\perp}$, so there exist unique elements $y' \in{\mathcal N}(B)$ and $y\in{\mathcal N}(B)^{\perp}$ such that $z=y'+y$. Thus $Ax=Bz=By'+By=By$. So we can define $C: \mathscr{H}\to \mathscr{H}$ by $Cx=y$. The operator $C$ is linear, bounded by the closed graph theorem, and $A=BC$.\\
(ii) $\Rightarrow$ (iii). It follows from $A=BC$ that $AA^*=BCC^*B^*\leq B\|CC^*\|B^*=\|C\|^2BB^*$.\\
(iii) $\Rightarrow$ (i). Let $D: {\mathcal R}(B)^* \to {\mathcal R}(A^*)$ be defined by $D(B^*x)=A^*x$. This gives us a well-defined linear mapping since
$\|D(B^*x)\|^2=\|A^*x\|^2=\langle AA^*x,x\rangle \leq k^2 \langle BB^*x,x\rangle=k^2\|B^*x\|^2$. The operator $D$ has a unique extension to $\overline{{\mathcal R}(B^*)}$. Setting $Dx=0$ for all $x\in \overline{{\mathcal R}(B^*)}^{\perp}={\mathcal N}(B)$, we reach to a bounded linear mapping $D$ such that $DB^*=A^*$. Thus $A=BD^*$, whence ${\mathcal R}(A) \subseteq {\mathcal R}(B)$.
\end{proof}
The Douglas theorem has been extensively applied in studies of quotients of operators \cite{IZU}, relationship between majorization and factorization in general $C^*$-algebras \cite{FIA}, lifting commuting operators \cite{DMP}, partial operator matrices \cite{HZ}, similarities and quasi similarities \cite{CF}, frames \cite{KHP}, and linear relations \cite{PS}.

An application of the Douglas theorem with $A^*$ and $|A|=(A^*A)^{1/2}$ instead of $A$ and $B$, respectively, gives the polar decomposition of $A=U|A|$, where $U=C^*$ is a partial isometry. In addition, $$\mathcal{R}(A)=\mathcal{R}((AA^*)^{1/2}). \quad (\diamond)$$

In the proof of (i) $\Rightarrow$ (ii), one can easily observe that $C=B_0^{-1}A$, where $B_0$ is the restriction of $B$ to ${\mathcal N}(B)^{\perp}$. From this one can immediately conclude that if both $A$ and $B$ are self-adjoint, then $C$ in (ii) is invertible if and only if $A$ and $B$ have the same range and nullity; cf. \cite[P.\,259, Corollary 1]{FW}. Thus two positive operators have the same range if and only if one of them is a factor of the other by an invertible operator. Therefore a positive operator $A$ has closed range if and only if $\mathcal{R}(A)=\mathcal{R}(A^{1/2})$, see \cite{FW}. Very recently, this fact is extended to the adjointable operators on Hilbert $C^*$-modules (see the last section for this terminology) by showing that $\mathcal{R}(A)$ is closed if and only if $\mathcal{R}(A)=\mathcal{R}(A^\alpha)$ for all (some) $0<\alpha \neq 1$ and this occurs if and only if $\mathcal{R}(A^\alpha)$ is closed for all (some) $0<\alpha \neq 1$; see \cite{VMX}.

Fillmore and Williams \cite{FW} extended the Douglas theorem to several operators as follows: They first use ($\diamond$) to show that $\mathcal{R}(A)+\mathcal{R}(B)=\mathcal{R}(\sqrt{AA^*+BB^*})$. To prove their claim, they passed to the direct sum $\mathscr{H}\oplus \mathscr{H}$ and considered $T=\left( \begin{array}{cc} A & B \\ 0 & 0 \\
 \end{array} \right) \in \mathbb{B}(\mathscr{H}\oplus \mathscr{H})$ and verified that
 $$(\mathcal{R}(A)+\mathcal{R}(B))\oplus 0 = \mathcal{R}(T)=\mathcal{R}((TT^*)^{1/2})=\mathcal{R}(\sqrt{AA^*+BB^*})\oplus 0.$$
Second, for $A, B_1, B_2 \in \mathbb{B}(\mathscr{H})$, they established the following conditions are equivalent:
\begin{enumerate}
\item[{\rm (i')}] $\mathcal{R}(A) \subseteq \mathcal{R}(B_1)+\mathcal{R}(B_2)$;
\item[{\rm (ii')}] $A=B_1C_1+B_2C_2$ for some $C_1, C_2 \in \mathbb{B}(\mathscr{H})$;
\item[{\rm (iii')}] $AA^* \leq k^2 (B_1B_1^*+B_2B_2^*)$ for some $k\geq 0$.
\end{enumerate}
To prove that conditions (i') and (iii') are equivalent, it is enough to consider the fact $\mathcal{R}(B_1)+\mathcal{R}(B_2)=\mathcal{R}(\sqrt{B_1B_1^*+B_2B_2^*})$ and use the Douglas theorem. Clearly (ii') implies (i'). Hence it is enough to show that (i') implies (ii'). To this end, set $S=\left( \begin{array}{cc} A & 0 \\ 0 & 0 \\ \end{array} \right)$ and $T=\left( \begin{array}{cc} B_1 & B_2 \\ 0 & 0 \\ \end{array} \right)$. Then $\mathcal{R}(S)\subseteq\mathcal{R}(T)$. Utilizing the Douglas theorem, we conclude that $S=TC$ for some $C=\left( \begin{array}{cc} C_1 & C_3 \\ C_2 & C_4 \\ \end{array} \right)\in \mathbb{B}(\mathscr{H}\oplus \mathscr{H})$.

A natural conjecture is that for $A_1, A_2, B_1, B_2 \in \mathbb{B}(\mathscr{H})$, the following conditions are equivalent:
\begin{enumerate}
\item[{\rm (i'')}] $\mathcal{R}(A_1)+ \mathcal{R}(A_2) \subseteq \mathcal{R}(B_1)+\mathcal{R}(B_2)$;
\item[{\rm (ii'')}] $A_k=B_1C_{1k}+B_2C_{2k}$ for some $C_{jk}\in \mathbb{B}(\mathscr{H})$ for $j,k=1,2$;
\item[{\rm (iii'')}] $A_1A_1^*+A_2A_2^* \leq k^2 (B_1B_1^*+B_2B_2^*)$ for some $k\geq 0$.
\end{enumerate}
Twice applications of the above result of Fillmore and Williams ensures that (i'') and (iii'') are equivalent and (iii'') implies that (ii'') but it is not clear whether (ii'') implies (i'').

In the case when we replace $\mathbb{B}(\mathscr{H})$ by a von Neumann algebra $\mathfrak{M}$, we can choose $C$ to be in $\mathfrak{M}$. Recall a result due to J. Cuntz stating that if $\mathfrak{A}$ is a $C^*$-algebra and $A,B \in \mathfrak{A}$ with $AA^* \leq BB^*$, then there exists a sequence $\{C_n\}$ in the unit ball of $\mathfrak{A}$ such that $\lim_n\|A-BC_n\|=0$; see \cite{FS}. Following the argument given by Fialkow and Salas \cite{FS}, assume that $A, B \in \mathfrak{M}$ and $AA^*\leq BB^*$. Then there is a sequence $\{C_n\}$ in the unit ball of $\mathfrak{M}$ such that $A=\lim_nBC_n$ in the norm topology. Since the unit ball of $\mathfrak{M}$ is compact in the weak operator topology, we may assume that $\{C_n\}$ converges to some $C$ in the unit ball of $\mathfrak{M}$. Thus $A=BC$.

One can prove this result in another way: It is easy to see that for any unitary $U$ in the double commutant of $\mathfrak{M}$, the operator $U^*CU$ satisfies the conditions of the Douglas theorem and so, by the uniqueness, $U^*CU=C$. Hence $C$ belongs to the double commutant of $\mathfrak{M}$, which is indeed itself $\mathfrak{M}$; see \cite{NOS}.

Let $A\in \mathbb{B}(\mathscr{H})$ be a positive operator. Then $\mathscr{H}$ can be endowed with the semi-inner product $\langle \xi,\eta \rangle_A:=\langle A\xi,\eta \rangle$ and we can investigate the adjoint operation: An operator $T\in \mathbb{B}(\mathscr{H})$ is called $A$-adjointable if there is an operator $T^\#\in \mathbb{B}(\mathscr{H})$ such that $\langle T\xi,\eta \rangle_A= \langle \xi, T^\#\eta \rangle_A$. The Douglas theorem ensures that $T$ admits an $A$-adjoint if and only if ${\mathcal R}(T^*A)\subseteq {\mathcal R}(A)$. Such adjoint $T^\#$ has many properties similar to those of the usual adjoint $T^*$; see \cite {ACG1}.

In 1971-1972, Leech \cite{LEE} showed that if $A$ and $B$ are contraction operators on a Hilbert space $\mathscr{H}$ that commute with a shift operator $S$, then $A = BC$ for some contraction operator $C$ on $\mathscr{H}$ that commutes with $S$ if and only if $AA^{*} \leq BB^{*}$. This result has been applied in indefinite inner product spaces.

Let $\mathbb{M}_n(\mathfrak{A})$ be the $n\times n$ matrices with entries in an algebra $\mathfrak{A}$ of operators acting on a Hilbert space $\mathscr{H}$. The algebra $\mathfrak{A}$ is said to have the Douglas property if for each $n$ and each $A, B \in \mathbb{M}_n(\mathfrak{A})$, there exists a contraction $C \in \mathbb{M}_n(\mathfrak{A})$ such that $AC = B$ if and only if $AA^* \leq BB^*$. Thus, the Douglas theorem and Leech's Theorem say that
$\mathbb{B}(\mathscr{H})$ and the algebra of analytic Toeplitz operators, respectively, have the Douglas property. The Douglas property for the some significant classes of $C^*$-algebras is studied by Fialkow and Salas \cite{FS}; in particular, they showed that every injective $C^*$-algebra has the Douglas property: Recall that a $C^*$-algebra $\mathfrak{A}$ of operators on a Hilbert space $\mathscr{H}$ is said to be injective if there exists a unital (contractive) completely positive map $\Phi: \mathbb{B}(\mathscr{H})\to \mathfrak{A}$ such that $\Phi(A_1CA_2)=A_1\Phi(C)A_2$ for all $A_1, A_2 \in\mathfrak{A}$ and all $C\in \mathbb{B}(\mathscr{H})$. Now, if $A, B\in\mathfrak{A}$ satisfy $AA^*\leq BB^*$, then by the Douglas theorem, there is a contraction $C\in \mathbb{B}(\mathscr{H})$ such that $A=BC$. Hence $A=\Phi(A)=\Phi(BC)=B\Phi(C)$ and $\|\Phi(C)\|\leq \|C\|\leq 1$.

Nayak \cite{NAY} shows that the $C^*$-algebra $C_0(X)$ of all continuous functions on $X$ vanishing at infinity has the Douglas property if and only if $X$ is sub-Stonean. He gives a counterexample of a $C^*$-algebra without the Douglas property:

Let $X=\{\frac{1}{n}: n \in \mathbb{N}\}\cup\{0\}$ be the compact topological subspace of the Euclidean space $\mathbb{R}$. Let $f(x)$ be defined by $x$ when $x=\frac{1}{2n}$ for some $n$ and $0$ otherwise, and let $g(x)=x$ for all $x\in X$. Clearly $\overline{f}f \leq \overline{g}g$, but if $f=gh$ for some function $h$ defined on $X$, then $h$ would not be continuous at zero. Thus $C(X)$ does not have the Douglas property.

In this expository article, we review some significant generalizations and applications of the celebrated Douglas theorem on equivalence of factorization, range inclusion, and majorization of operators. We then apply it to find a characterization of the positivity $2\times 2$ block matrices of operators on Hilbert spaces and finally describe the nature of such matrices and provide several ways for showing their positivity.

%---------------------------------------------------------------------------------------%

\section{Douglas theorem in the setting of Banach spaces}

In 1973, Embry \cite{EMB} proved that if $A$ and $B$ are bounded linear operators on a Banach space $X$, then the following statements are equivalent: (i) $A=CB$ for some bounded linear operator $C: {\mathcal R}(B) \to X$; (ii) $\|Ax\|\leq k\|Bx\|$ for some $k\geq 0$ and all $x\in X$; (iii) ${\mathcal R}(A^*) \subseteq {\mathcal R}(B^*)$.

In 1978, Bouldin \cite{BOU} showed that if $X$ is the Banach space $c_0$ consisting of complex sequences converging to zero with the usual Schauder basis $\{e_n\}$, and operators $A$ and $B$ are defined on $X$ by $A(e_k)=0$ for $k\neq 1$ and $Ae_1=(2^{-1}, 2^{-2}, \ldots)$ and $B(x_1, x_2, \ldots)=(2^{-1}x_1, 2^{-2}x_2, \ldots)$, respectively, then $\|A^\prime f\|\leq c\|B^\prime f\|$ for all $f\in (c_0)^\prime$, where $A^\prime$ denotes the dual (Banach adjoint) of $A$ but not $\mathcal{R}(A) \subseteq \mathcal{R}(B)$.

In 1991, Andrew and Patterson \cite{AP} showed that on many of the classical Banach spaces including $L_p\,\,(1 \leq p<\infty, p\neq 2)$, the sequence spaces $c_0$ and $\ell_p\,\,(1 <p<\infty, p\neq 2)$, and the function spaces $C(\Omega)$ with $\Omega$ compact and metric, there exist operators $A$ and $B$ with $\mathcal{R}(A) \subseteq \mathcal{R}(B)$ yet $A\neq BC$ for any bounded linear operator $C$.

In 2004, Barnes \cite{BAR} presented a comprehensive picture of relationships among the concepts of majorization, range inclusion, and factorization in the setting of bounded linear operators on Banach spaces.

Douglas in \cite{DOU} also obtained a version of his theorem for unbounded operators acting on Hilbert spaces. Recently, Forough \cite{FOR} extended the Douglas theorem to the context of closed densely defined operators on Banach spaces.

In general, the Douglas theorem does not hold in Krein spaces (not even in the finite dimensional case). A version of the Douglas theorem for Krein spaces is given by Rodman \cite{ROD}.

The authors of \cite{WZ} considered the Douglas theorem in the setting of non-Archimedean Banach spaces over complete nontrivially non-Archimedean-valued fields.

%---------------------------------------------------------------------------------------%

\section{Douglas theorem and operator equations}

There are significant applications of the Douglas theorem in study of solutions of operator equations like $AX = B$. Nakamoto \cite{NAK} used the Douglas theorem to prove an existence theorem for solutions of the operator equation $XAX = B$. Arias, Corach, and Gonzalez \cite{ACG2} introduced the notion of reduced solution which is a generalization of the concept of Douglas solution.

The notion of generalized inverse plays an essential role in the topic. An operator $A \in \mathbb{B}(\mathscr{K},\mathscr{H})$ is said to be a generalized inverse of an operator $A^\prime \in \mathbb{B}(\mathscr{H},\mathscr{K})$ if $AA^\prime A=A$. It is known that $A \in \mathbb{B}(\mathscr{H},\mathscr{K})$ has a generalized inverse if and only if $\mathcal{R}(A)$ is closed. If so, one can choose $A^\prime$ in such a way that $A^\prime AA^\prime=A^\prime$, $(AA^\prime)^*=AA^\prime$, and $(A^\prime A)^*=A^\prime A$. Such a generalized inverse is unique, denoted by $A^\dagger$, and is called the Moore--Penrose inverse of $A$.

Let $A\in \mathbb{B}(\mathscr{H},\mathscr{K}), B\in \mathbb{B}(\mathscr{G},\mathscr{L}), C \in \mathbb{B}(\mathscr{G},\mathscr{K})$. If one of the subspaces ${\mathcal R}(A)$, $R(B)$, or ${\mathcal R}(C)$ is closed, then $AXB=C$ is solvable if and only if ${\mathcal R}(C)\subseteq{\mathcal R}(A)$ and ${\mathcal R}(C^*)\subseteq{\mathcal R}(B^*)$ if and only if $ATCSB=C$ for all generalized inverses $T,S$ of $A,B$, respectively, see \cite{AG} as well as \cite{LDLY}.

There are some applications of the Douglas theorem in establishing some Putnam--Duglede type theorems. Let $A, B \in \mathbb{B}(\mathscr{H})$ and $\delta_{A, B}(X):=AX-XB$ and $\Delta_{A, B}(X):=AXB-X$ for $X\in \mathbb{B}(\mathscr{H})$. An operator $A\in \mathbb{B}(\mathscr{H})$ satisfies the Putnam--Fuglede property $\delta$ (the Putnam--Fuglede property $\Delta$, respectively), if for every isometry $V\in \mathbb{B}(\mathscr{H})$ for which the equation $\delta_{A, V^*}(X)=0$ ($\Delta_{A,V^*}(X)=0$, respectively) has a non-trivial solution $X\in \mathbb{B}(\mathscr{H})$, the solution $X$ fulfills $\delta_{A^*, V}(X)=0$ ($\Delta_{A^*,V}(X)=0$, respectively). Duggal and Kubrusly \cite{DK} used the Douglas theorem to show that $A$ satisfies the Putnam--Fuglede property $\delta$ if and only if it satisfies the Putnam--Fuglede property $\Delta$; see also \cite{MN}.

%---------------------------------------------------------------------------------------%

\section {Douglas theorem in the framework of Hilbert $C^*$-modules}

Hilbert $C^*$-modules are generalizations of Hilbert spaces by allowing inner products to take values in some $C^{*}$-algebras instead of the field of complex numbers. More precisely, a Hilbert module over a $C^*$-algebra $\mathfrak{A}$ is a right $ \mathfrak{A}$-module equipped with an $ \mathfrak{A}$-valued inner product $\langle \cdot, \cdot \rangle: \mathscr{H} \times \mathscr{H} \to \mathfrak{A}$ such that $\mathscr{H}$ is complete with respect to the induced norm defined by $\| x\| = \| \langle x,x\rangle \|^{\frac{1}{2}}$ for $x\in \mathscr{H}$. A closed submodule $\mathscr{F}$ of a Hilbert $\mathfrak{A}$-module $\mathscr{H}$ is said to be orthogonally complemented if $\mathscr{H}=\mathscr{F}\oplus \mathscr{F}^\bot$, where $\mathscr{F}^\bot=\left\{x\in \mathscr{H} : \langle x,y\rangle=0\ \mbox{for any }\ y\in \mathscr{F}\right\}$. Let $\mathfrak{L}( \mathscr{H}, \mathscr{K})$ be the set of maps $A: \mathscr{H} \to \mathscr{K}$ between Hilbert $\mathfrak{A}$-modules for which there is a map $A^{*}: \mathscr{K}\to \mathscr{H}$, called the adjoint operator of $A$, such that $\langle Ax,y\rangle= \langle x,A^{*}y\rangle, \mbox{for any $x\in \mathscr{H}$ and $y\in \mathscr{K}$}$. In general, an operator may be not adjointable. We use $\mathfrak{L}( \mathscr{H})$ to denote the $C^*$-algebra $\mathfrak{L}( \mathscr{H}, \mathscr{H})$. If $A\in \mathfrak{L}( \mathscr{H}, \mathscr{K})$ does not have closed range, then neither $ {\mathcal N}(A) $ nor $ \overline{{\mathcal R}(A)} $ need to be orthogonally complemented. In addition, if $A\in \mathfrak{L}( \mathscr{H}, \mathscr{K})$ and $ \overline{{\mathcal R}(A^*)} $ is not orthogonally complemented, then it may happen that ${\mathcal N}(A)^{\bot}\neq \overline{{\mathcal R}(A^*)}$; see \cite{LAN}. The above facts show that the theory Hilbert $C^*$-modules are much different and more complicated than that of Hilbert spaces.

A generalization of the Douglas theorem to the framework of Hilbert $C^*$-modules in which we do not assume that $\overline{{\mathcal R}(T^*)}$ is orthogonally complemented reads as follows.
\begin{theorem}\label{thm:fang s result} \cite[Theorem 1.1]{FYY} \cite[Corollary 2.5]{FMX} Let $\mathfrak{A}$ be a $C^*$-algebra, $\mathscr{E},\mathscr{H}$ and $\mathscr{K}$ be Hilbert $\mathfrak{A}$-modules. Let $T\in {\mathfrak{L}}(\mathscr{E},\mathscr{K})$ and $T^\prime\in {\mathfrak{L}}(\mathscr{H},\mathscr{K})$. Then the following statements are equivalent:
\begin{enumerate}
\item[{\rm (i)}] $T^\prime (T^\prime)^*\le \lambda TT^*$ for some $\lambda>0$;
\item[{\rm (ii)}]There exists $\mu>0$ such that $\Vert (T^\prime)^*z\Vert\le \mu \Vert T^*z\Vert$, for any $z\in \mathscr{K}$.
\end{enumerate}
\end{theorem}
It is notable that $T^\prime (T^\prime)^*\le \lambda TT^*$ for some $\lambda>0$ does not ensure ${\mathcal R}(T^\prime)\subseteq {\mathcal R}(T) $.
\begin{theorem} \label{thm:two orthogonally complemented conditions} \cite[Theorem 1.1]{FYY} \cite[Theorem 3.2]{FMX} Let $\mathscr{E}, \mathscr{K}$ be Hilbert $\mathfrak{A}$-modules and $T$ be in $\mathfrak{L}(\mathscr{E},\mathscr{K})$. Then the following statements are equivalent:
\begin{enumerate}
\item[{\rm (i)}] $\overline{{\mathcal R}(T^*)}$ is orthogonally complemented;
\item[{\rm (ii)}] For any Hilbert $\mathfrak{A}$-module $\mathscr{H}$ and any $T^\prime\in {\mathfrak{L}}(\mathscr{H},\mathscr{K})$, the equation
$T^\prime =TX$ for $X\in {\mathfrak{L}}(\mathscr{H},\mathscr{E})$ is solvable whenever ${\mathcal R}(T^\prime)\subseteq {\mathcal R}(T)$;
\item[{\rm (iii)}] The equation
$S=TX$ for $X \in {\mathfrak{L}}(\mathscr{G},\mathscr{E})$ is solvable, where $\mathscr{G}=\overline{\mathcal{R}(T^*)}$ and $S$ is the restriction of $T$ on $\mathscr{G}$.
\end{enumerate}
If condition (i) is satisfied and $T^\prime\in {\mathfrak{L}}(\mathscr{H},\mathscr{K})$ is such that $\mathcal{R}(T^\prime)\subseteq {\mathcal R}(T)$, then there exists a unique $D\in {\mathfrak{L}}(\mathscr{H},\mathscr{E})$ which satisfies
\begin{equation}\label{equ:uniqueness of D}T^\prime=TD\ \mbox{and}\ {\mathcal R}(D)\subseteq {\mathcal N}(T)^\bot.\end{equation}
In this case,
\begin{equation}\label{equ:additional conditions of D}\Vert D\Vert^2=\inf\left\{\lambda : T^\prime(T^\prime)^*\le \lambda TT^*\right\}.\end{equation}
\end{theorem}
If $T\in \mathfrak{L}(\mathscr{E},\mathscr{K})$ is such that $\overline{{\mathcal R}(T^*)}$ is not orthogonally complemented, then it is immediately follows from Theorem \ref{thm:two orthogonally complemented conditions} that the equation $S=TX$ for $X \in {\mathfrak{L}}(\mathscr{G},\mathscr{E})$ is unsolvable. The general form of the Douglas theorem does not hold in the setting of Hilbert $C^*$-modules as shown in \cite[Proposition 3.4]{FMX} that there exist Hilbert $\mathfrak{A}$-modules $\mathscr{H},\mathscr{K}$, and $T\in {\mathfrak{L}}(\mathscr{H},\mathscr{K})$ and $T^\prime\in {\mathfrak{L}}(\mathscr{K})$ such that ${\mathcal R}(T^\prime)$ is contained in ${\mathcal R}(T)$ properly, whereas the equation $TX=T^\prime$ for $X\in\mathfrak{L}(\mathscr{K},\mathscr{H})$ has no solution.

Another application of the Douglas theorem is in the study of the parallel sum. Let $\mathscr{H}$ be a Hilbert space and $A,B\in\mathbb{B}(\mathscr{H})$ be positive. Put
\begin{equation*}T=\left(
  \begin{array}{cc}
   0 & 0 \\
   A^{1/2} & B^{1/2} \\
  \end{array}\right)\in \mathbb{B}(\mathscr{H}\oplus \mathscr{H}).
\end{equation*}
A direct application of ($\diamond$) yields
$$\mathcal{R}(A^{1/2})+\mathcal{R}(B^{1/2})=\mathcal{R}((A+B)^{1/2})\quad (\diamond \diamond).$$
It follows from ($\diamond \diamond$) and the Douglas theorem that
there are uniquely determined operators $C, D\in \mathbb{B}(\mathscr{H})$ such that
\begin{eqnarray*}A^{1/2}=(A+B)^{1/2}C \ \mbox{and}\ \mathcal{R}(C)\subseteq \mathcal{N}((A+B)^{1/2})^{\perp},\\
B^{1/2}=(A+B)^{1/2}D \ \mbox{and}\ \mathcal{R}(D)\subseteq \mathcal{N}((A+B)^{1/2})^{\perp}.
\end{eqnarray*}
Following \cite[P.\,277]{FW} the parallel sum of two positive operators $A$ and $B$ is defined by
\begin{equation*}A:B=A^{1/2}C^*DB^{1/2}.
\end{equation*}

Note that if $A+B$ is Moore-Penrose invertible (or equivalently, $\mathcal{R}(A+B)$ is closed), then it is easy to verify that $$C=\left((A+B)^\dag\right)^\frac12\cdot A^\frac12\ \mbox{and}\ D=\left((A+B)^\dag\right)^\frac12\cdot B^\frac12,$$
so in this case we have $A^{1/2}C^*DB^{1/2}=A(A+B)^\dag B$. Thus, the term of the parallel sum for positive operators $A$ and $B$ was generalized in \cite{FW}
without any restrictions on $\mathcal{R}(A+B)$.

%%%%%%%%%%%%%%%%%%%%%%%%%%%%%%%%%%%%%%%%%%%%%%%%%%%%%%%%%%%%%%%%%%

%%%%%%%%%%%%%%%%%%%%%%%%%%%%%%%%%%%%%%%%%%%%%%%%%%%%%%%%%%%%%%%%%%

%%%%%%%%%%%%%%%%%%%%%%%%%%%%%%%%%%%%%%%%%%%%%%%%%%%%%%%%%%%%%%%%%%

%%%%%%%%%%%%%%%%%%%%%%%%%%%%%%%%%%%%%%%%%%%%%%%%%%%%%%%%%%%%%%%%%%

%%%%%%%%%%%%%%%%%%%%%%%%%%%%%%%%%%%%%%%%%%%%%%%%%%%%%%%%%%%%%%%%%%

%%%%%%%%%%%%%%%%%%%%%%%%%%%%%%%%%%%%%%%%%%%%%%%%%%%%%%%%%%%%%%%%%%

\section{Positivity of $2\times 2$ block matrices of operators}

The algebra of all $2\times 2$ matrices with entries in $\mathbb{B}(\mathscr{H})$ is a typical example of the $C^*$-algebra $\mathbb{M}_2({\mathfrak A})$. Clearly, if $C^*$-algebras ${\mathfrak A}$ and ${\mathfrak B}$ are $*$-isomorphic, then $\mathbb{M}_2({\mathfrak A})$ and $\mathbb{M}_2({\mathfrak B})$ are also $*$-isomorpic. But there exist two non-isomorphic unital $C^*$-algebra ${\mathfrak A}$ and ${\mathfrak B}$ such that ${\mathfrak A}\simeq \mathbb{M}_2({\mathfrak A})\simeq \mathbb{M}_2({\mathfrak B})$. For example, consider ${\mathfrak A}=\{T\oplus T; T \in \mathbb{B}(\mathscr{H})\}+\mathbb{K}(\mathscr{H}\oplus \mathscr{H})$ and ${\mathfrak B}=\{T \oplus T\oplus 0 ; T \in \mathbb{B}(\mathscr{H}), 0 \in \mathbb{B}(\mathscr{H}_0)\}+\mathbb{K}(\mathscr{H}\oplus \mathscr{H}\oplus \mathscr{H}_0)$ where $\mathscr{H}$ is a separable infinite dimentional Hilbert space, $\mathscr{H}_0$ is one dimensional, and $\mathbb{K}(\mathscr{H}_1)$ denotes the algebra of compact linear operators on the Hilbert space $\mathscr{H}_1$; see \cite{PLA}.\\
If ${\mathfrak A}$ is a $C^*$-algebra containing $\mathbb{M}_2(\mathbb{C})$ as a unital $C^*$-subalgebra, then ${\mathfrak A}\simeq \mathbb{M}_2({\mathfrak B})$ for some $C^*$-algebra ${\mathfrak B}$, say ${\mathfrak B}=\begin{pmatrix}1 & 0 \\0 & 0\end{pmatrix}{\mathfrak A} \begin{pmatrix}1 & 0 \\0 & 0\end{pmatrix}$. Moreover, a unital $C^*$ algebra ${\mathfrak A}$ is $*$-isomorphic to $\mathbb{M}_2({\bf C})$ iff there exists a projection $p \in {\mathfrak A}$ such that \[p{\mathfrak A}p= \mathbb{C}p, (1-p){\mathfrak A}(1-p)= \mathbb{C}(1-p), (1-p){\mathfrak A}p \neq 0, p{\mathfrak A}(1-p) \neq 0;\] see \cite{MOS2}.\\
The theory of $2 \times 2$ operator matrices is not trivial \cite[Problem 56]{HAL}: If $R,S,T \in \mathbb{B}(\mathscr{H})$ and both $R,T$ are invertible, then $U=\left (\begin{array}{cc}R&S\\0&T \end{array}\right )$ is invertible and $U^{-1}=\left ( \begin{array}{cc}R^{-1}&-R^{-1}ST^{-1}\\0&T^{-1} \end{array}\right)$. The converse is not true, in general. For example, take $\mathscr{H}=\ell^2,~T(x_1,x_2,\cdots)=(x_2,x_3,\cdots), R(x_1,x_2,\cdots)(0,x_1,\cdots)$ and $S(x_1,x_2, \cdots)=(x_1,0,0,\cdots)$. Then $\left (\begin{array}{cc}R&S\\0&T \end{array}\right)$ is invertible but neither $R$ nor $T$ is invertible.

There are several techniques utilizing $2 \times 2$ matrices over a $C^*$-algebra ${\mathfrak A}$, which helps us to solve some problems which could not be treated easily in the $C^*$-algebra ${\mathfrak A}$ itelf. They are used to reduce a problem to a simple one and to pass from a concept to other appropriate ones. For instance, we refer the reader to the Kaplansky density theorem and the Putnam-Fuglede theorem, the study of derivations and similarity problems; the investigation of completely bounded maps, counterexamples, and so on; see \cite{MOS1}.

We adopt some notation, facts and techniques from some books, say \cite{BHA1, MUR, PAU}.

Following we state the first characterization of the positivity of a $2\times 2$ block matrix $M$ defined as:
\begin{equation}\label{defn of M}M=\begin{pmatrix}A & X \\X^* & B\end{pmatrix}\in\mathbb{B}({\mathscr H}\oplus {\mathscr H}),\ \mbox{where $A, B\in {\mathbb B}({\mathscr H})$}.\end{equation}
\begin{theorem}\label{nice1}
Let $M$ be given by \eqref{defn of M} such that
$A \in {\mathbb B}_{++}({\mathscr H})$. Then
$M$ is positive if and only if $B\geq X^*A^{-1}X$.
\end{theorem}
\begin{proof} We provide three proofs with different methods.

\textbf{The first proof} (\cite[Theorem 1.3.3]{BHA1}):  Put $C=\begin{pmatrix}
I & O \\
-X^*A^{-1} & I
\end{pmatrix}$. Then $C$ is invertible such that
\begin{equation*}
A \oplus (B-X^*A^{-1}X)=C\cdot M\cdot C^*,
\end{equation*}
hence we may use the fact that ``if $0 \leq D_1\leq D_2$, then $Z^*D_1Z \leq Z^*D_2Z$ for all $Z\in \mathbb{B}(\mathscr{H})$'' to conclude that $M\ge 0\Longleftrightarrow B\ge X^*A^{-1}X$.

\textbf{The second proof:} Suppose that $M\ge 0$. Let
\begin{equation*}M^\frac12=\left(
       \begin{array}{cc}
        Y_{11} & Y_{12} \\
        Y_{21} & Y_{22} \\
       \end{array}
       \right)=\left(
         \begin{array}{cc}
          Y_1 & Y_2 \\
         \end{array}
         \right),\ \mbox{where $Y_1=\left(
                \begin{array}{c}
                 Y_{11} \\
                 Y_{21} \\
                \end{array}
                \right)$ and $Y_2=\left(
                \begin{array}{c}
                 Y_{12} \\
                 Y_{22} \\
                \end{array}
                \right)$}.
         \end{equation*}
Then
\begin{equation*}M=\left(
      \begin{array}{c}
      Y_1^* \\
      Y_2^* \\
      \end{array}
     \right)\cdot (Y_1,Y_2)=\left(
           \begin{array}{cc}
            Y_1^*Y_1 & Y_1^*Y_2 \\
            Y_2^*Y_1 & Y_2^*Y_2 \\
           \end{array}
           \right),
\end{equation*}
hence $A=Y_1^*Y_1, X=Y_1^*Y_2$ and $B=Y_2^*Y_2$. Let
$$Z=Y_1\cdot A^{-1}\cdot Y_1^*=Y_1\cdot (Y_1^*Y_1)^{-1}\cdot Y_1^*.$$
Then $Z$ is positive since $A^{-1}$ is, and $Z^2=Z$, therefore $Z$ is a projection, which gives $I-Z\ge 0$. It follows that
$B-X^*A^{-1}X=Y_2^*(I-Z)Y_2\ge 0$.

Conversely, assume that $B\ge X^*A^{-1}X$. Let
\begin{equation}\label{defn of Y}Y=(Y_1,Y_2)\ \mbox{with $Y_1=A^\frac12$ and $Y_2=A^{-\frac12}X$}.\end{equation} Then
$M=Y^*Y+0\oplus (B-X^*A^{-1}X)\ge 0$.

\textbf{The third proof:} Assume that $M\ge 0$. Then for any $x,y\in {\mathscr H}$, we have
$$\left\langle M(x,-y)^T, (x,-y)^T\right\rangle\geq 0.$$
Hence,
$$\langle Ax,x\rangle - 2 \textrm{Re} \langle Xy,x\rangle +\langle By,y\rangle \geq 0.$$
Replacing $x$ by $A^{-1}Xy$ and noting that $A^{-1}$ is positive, we get
\begin{eqnarray*}
\langle By,y\rangle &\geq& 2 \textrm{Re} \langle Xy,A^{-1}Xy\rangle-\langle AA^{-1}Xy,A^{-1}Xy\rangle\\
&=& 2 \textrm{Re} \langle X^*A^{-1}Xy,y\rangle-\langle X^*A^{-1}Xy,y\rangle\\
&=& \langle X^*A^{-1}Xy,y\rangle
\end{eqnarray*}
since $\langle X^*A^{-1}Xy,y\rangle \geq 0$. Thus $B \geq X^*A^{-1}X$.

Conversely, assume that $B \geq X^*A^{-1}X$ Let $Y$ be defined by \eqref{defn of Y}. Then for any $x,y\in {\mathscr H}$, we have
\begin{align*}\langle M(x,y)^T,(x,y)^T\rangle=\|Y_1x+Y_2y\|^2+\langle (B-X^*A^{-1}X)y, y\rangle\ge 0,
\end{align*}
hence $M\ge 0$.
\end{proof}

\begin{remark} Let $M$ be given by \eqref{defn of M} such that $M\ge 0$. It is interesting to find out an explicit formula for $M^\frac12$ under certain restrictions on $X$ and $B$. In the matrix case, a generalization of the preceding theorem in term of the generalized inverse $A^\dag$ can be found in \cite{FUJ-1,FUJ-2}.
\end{remark}

\begin{corollary}\label{cnice1}
Let $A \in {\mathbb B}_{++}({\mathscr H})$ and $X\in {\mathbb B}({\mathscr H})$. Then
$$X^*A^{-1}X=\min\left\{B \in {\mathbb B}_+({\mathscr H}):
\begin{pmatrix}
A & X \\
X^* & B
\end{pmatrix}\geq0 \right\},$$
or equivalently
$$B - X^*A^{-1}X=\sup\left\{ Y\geq 0: \begin{pmatrix}0 & 0 \\0 & Y\end{pmatrix} \leq \begin{pmatrix}A & X \\X^* & B\end{pmatrix} \right\}.$$
\end{corollary}
%%%%%%%%%%%%%%%%%%%%%%%%%%%%%%%%%%%%%%%%%%%%%%%%%%%%%%%%%%%%%%%%%%
\begin{remark}\label{rem:exchange with 1 and 2}
Let $M$ be given by \eqref{defn of M} such that $B \in {\mathbb B}_{++}({\mathscr H})$. Then Theorem \ref{nice1} turns to
 \begin{equation}\label{equ:exchange with 1 and 2}M\geq 0 \ \Longleftrightarrow\ A\geq XB^{-1}X^*.\end{equation}
 To see this, note that the matrix $V=\begin{pmatrix}0 & I \\I & 0\end{pmatrix}$ is unitary. Therefore the block matrix $M$ is positive if and only if $VMV^*=\begin{pmatrix}B & X^* \\X & A\end{pmatrix}$ is positive. Theorem \ref{nice1} now gives the equivalent condition $A\geq XB^{-1}X^*$.
 
An application of \eqref{equ:exchange with 1 and 2} can be found in a recent work of Najafi \cite{NAJ}.
\end{remark}

%%%%%%%%%%%%%%%%%%%%%%%%%%%%%%%%%%%%%%%%%%%%%%%%%%%%%%%%%%%%%%%%%%%%
\begin{corollary}\label{co1}
 Let $M$ be given by \eqref{defn of M} such that $B \in {\mathbb B}_{++}({\mathscr H})$ and $M\ge 0$. Then there exists a partial isometry $U$ such that $|X|\leq |A^{1/2}UB^{1/2}|$. If furthermore
 $\overline{\mathcal{R}(X)}=\overline{\mathcal{R}(X^*)}=\mathscr{H}$, then $U$ can chosen to be a unitary.
 \end{corollary}
\begin{proof} Let $T=XB^{-\frac12}$ and $T=U|T|$ be the polar decomposition. Then $|T|=U^*T$, hence
$|T|^2=U^*T\cdot T^*U$, so by Remark~\ref{rem:exchange with 1 and 2} we have
\begin{align*}|X|^2=&X^*X=B^\frac12\cdot |T|^2 \cdot B^\frac12=B^\frac12\cdot U^*TT^*U \cdot B^\frac12\\
=&B^\frac12 U^*\cdot XB^{-1}X^*\cdot (B^\frac12 U^*)^*\\
\le& B^\frac12 U^*\cdot A\cdot (B^\frac12 U^*)^*=|A^{1/2}UB^{1/2}|^2,
\end{align*}
which gives $|X|\leq |A^{1/2}UB^{1/2}|$ by the operator monotonicity of $t\mapsto t^{1/2}$.

Since $\mathcal{R}(U)=\overline{\mathcal{R}(T)}=\overline{\mathcal{R}(X)}$ and
$\mathcal{R}(U^*)=\overline{\mathcal{R}(T^*)}=\mathscr{H}\Longleftrightarrow\overline{\mathcal{R}(X^*)}=\mathscr{H}$, we know that
$U$ is a unitary if and only if $\overline{\mathcal{R}(X)}=\overline{\mathcal{R}(X^*)}=\mathscr{H}$.
\end{proof}

%%%%%%%%%%%%%%%%%%%%%%%%%%%%%%%%%%%%%%%%%%%%%%%%%%%%%%%%%%%%%%%%%%%%%%%%%%
\begin{remark}
It should be noted that the converse of Corollary \ref{co1} does not hold: $|X|\leq |A^{1/2}UB^{1/2}|$ does not imply the positivity of the block matrix $M$. For example, take $X=I$, $A=\left(\begin{array}{cc} 2 & 0 \\ 0 & 1\end{array}\right)$ and $B^{-1}=\left(\begin{array}{cc} 1 & 0\\ 0 & 2\end{array}\right)$ and $U=\left(\begin{array}{cc} 0 & 1 \\ 1 & 0\end{array}\right)$.
\end{remark}
%%%%%%%%%%%%%%%%%%%%%%%%%%%%%%%%%%%%%%%%%%%%%%%%%%%%%%%%
A necessary condition for the positivity of $M$ can be obtained by using Douglas theorem as follows: 
\begin{lemma}\cite[Lemma~6]{FFN}\label{lem:plus1}
 Let $M$ be given by \eqref{defn of M} such that $M\ge 0$. Then
$$\mathcal{R}(X)\subseteq \mathcal{R}(A^\frac12)\ \mbox{and}\ \mathcal{R}(X^*)\subseteq \mathcal{R}(B^\frac12).$$
\end{lemma}
\begin{proof} Since $M\ge 0$, we have $A,B \in {\mathbb B}_{+}({\mathscr H})$. Let $M^\frac12=\left(
                      \begin{array}{cc}
                       Y & W \\
                       W^* & Z \\
                      \end{array}
                      \right)$. Then direct computation yields
$$M=\left(
  \begin{array}{cc}
  Y^2+ZZ^* & YW+WZ \\
  W^*Y+ZW^* & W^*W+Z^2 \\
  \end{array}
 \right),$$
hence
$$A=Y^2+WW^*\ \mbox{and}\ X=YW+WZ,$$
which indicate by Douglas theorem that $\mathcal{R}(A^\frac12)$ contains both $\mathcal{R}(Y)$ and $\mathcal{R}(W)$. As a result, we have 
$\mathcal{R}(X)\subseteq \mathcal{R}(A^\frac12)$ since it is clear that $\mathcal{R}(X)\subseteq \mathcal{R}(Y)+\mathcal{R}(W)$.
Similarly, we have $\mathcal{R}(X^*)\subseteq \mathcal{R}(B^\frac12)$. 
\end{proof}

Based on the preceding lemma, another characterization of the positivity of $M$ can be derived as follows:
\begin{theorem}\cite[Theorem~7]{FFN}\label{nice3}
 Let $M$ be given by \eqref{defn of M} such that $A,B \in {\mathbb B}_{+}({\mathscr H})$. Then
$M\ge 0$ if and only if there is an operator $C\in {\mathbb B}({\mathscr H})$ such that $X=C^*B^{1/2}$ and $C^*C \leq A$.
\end{theorem}
\begin{proof}
(i) $\Rightarrow$ (ii). Suppose that $M\ge 0$. By Lemma~\ref{lem:plus1} and the Douglas theorem, there exists $C\in {\mathbb B}({\mathscr H})$
 such that 
 $$X^*=B^\frac12 C\ \mbox{and}\ \mathcal{R}(C)\subseteq \overline{\mathcal{R}(B^\frac12)}=\overline{\mathcal{R}(B)},$$
 hence $PC=C$ and $C^*P=C^*$, where $P$ is the projection from $H$ onto $\overline{\mathcal{R}(B)}$. 
 
 Given any $y\in\mathcal{R}(B)$ and $z\in {\mathscr H}$, we have $y=Bu$ for some $u\in {\mathscr H}$. Put $x=B^\frac12 u$. Then $y=B^\frac12 x$ and thus
 \begin{align*}\left\langle \left(
        \begin{array}{cc}
        A & C^* \\
        C & P \\
        \end{array}
       \right)\left(
          \begin{array}{c}
          z \\
          y \\
          \end{array}
         \right), \left(
          \begin{array}{c}
          z \\
          y \\
          \end{array}
         \right)\right\rangle=&\langle Az,z\rangle+\langle X x,z\rangle++\langle X^*z,x\rangle+\langle Bx,x\rangle\\
       =&\left\langle M\left(
            \begin{array}{c}
            z \\
            x \\
            \end{array}
           \right),\left(
            \begin{array}{c}
            z \\
            x \\
            \end{array}
           \right)\right\rangle\ge 0.
 \end{align*}
 It follows that $\left(
        \begin{array}{cc}
        A & C^* \\
        C & P \\
        \end{array}
       \right)\ge 0$, and so
 $$\left(
  \begin{array}{cc}
  A-C^*C & 0 \\
  0 & P \\
  \end{array}
 \right)=\left(
    \begin{array}{cc}
    I & -C^* \\
    0 & I \\
    \end{array}
   \right)\left(
     \begin{array}{cc}
      A & C^* \\
      C & P \\
     \end{array}
     \right)\left(
       \begin{array}{cc}
        I & 0 \\
        -C & I \\
       \end{array}
       \right)\ge 0,
 $$
 that is, $A\ge C^*C$, as required.

(ii) $\Rightarrow$ (i). It follows from the identity
$$M=\begin{pmatrix}A-C^*C& 0 \\0 & 0\end{pmatrix}+\begin{pmatrix}C & B^{1/2} \\0 & 0\end{pmatrix}^*\begin{pmatrix}C & B^{1/2} \\0 & 0\end{pmatrix}.$$
\end{proof}

Similarly, we have the following corollary:
\begin{corollary}\label{nice3-cor}
 Let $M$ be given by \eqref{defn of M} such that $A,B \in {\mathbb B}_{+}({\mathscr H})$. Then
$M\ge 0$ if and only if there is an operator $C\in {\mathbb B}({\mathscr H})$ such that $X=B^{1/2}C$ and $C^*C \leq B$.
\end{corollary}

Recall that an operator $K\in {\mathbb B}({\mathscr H})$ is a contraction if $\|K\|\leq 1$. We can also provide another characterization of the positivity of $M$ in terms of contractive operators. To this end, we need firstly a lemma as follows (see \cite[Proposition 1.3.1]{BHA1} for another proof for matrices):
\begin{lemma}\label{nice2}
An operator $K\in {\mathbb B}({\mathscr H})$ is contractive if and only if $\begin{pmatrix}
I & K \\
K^* & I
\end{pmatrix} \geq 0.$
\end{lemma}
\begin{proof}
The operator $K$ is a contraction if and only if $\|K^*K\|=\|K\|^2 \leq 1$. The later is equivalent to $0\leq t\leq 1$ for all $t \in {\rm sp}(K^*K))$. This fact, by the functional calculus for the self-adjoint operator $K^*K$, is equivalent to $K^*K \leq I$. Eventually, it follows from Theorem \ref{nice1} that $K^*K \leq I$ holds if and only if $\begin{pmatrix}
I & K \\
K^* & I
\end{pmatrix} \geq 0$.
\end{proof}

The next theorem is due to T. Ando \cite{ANDO}. Based on Corollary~\ref{nice3-cor}, an alternative proof of this theorem can be found in \cite[Theorem~8]{FFN}. A proof for matrices is stated in \cite[Theorem 7.7.9]{HJ}; see also \cite[Proposition 1.3.2]{BHA1} by which a proof of Theorem \ref{nice1} is given.
\begin{theorem}\label{contr}
Let $A, B\in {\mathbb B}_{+}({\mathscr H})$ and $X\in {\mathbb B}({\mathscr H})$.
Then the following are equivalent:
\begin{enumerate}
\item[(i)] The block matrix $M$ defined by \eqref{defn of M} is positive;
\item[(ii)] There exists a contraction $K$ such that $X=A^{1/2}KB^{1/2}$.
\end{enumerate}
\end{theorem}
\begin{proof}
(i) $\Rightarrow$ (ii). First assume that $A>0$ and $B>0$. By Theorem~\ref{nice1}, we have $B\geq X^*A^{-1}X$. Hence
\begin{eqnarray*}
I \geq B^{-1/2} X^*A^{-1}XB^{-1/2}=(A^{-1/2}XB^{-1/2})^*(A^{-1/2}XB^{-1/2})\,.
\end{eqnarray*}
Hence $\|A^{-1/2}XB^{-1/2}\|^2=\|(A^{-1/2}XB^{-1/2})^*(A^{-1/2}XB^{-1/2})\| \leq \|I\|=1$. Therefore the operator $K=A^{-1/2}XB^{-1/2}$ is a contraction and $$A^{1/2}KB^{1/2}=A^{1/2}(A^{-1/2}XB^{-1/2})B^{1/2}=X.$$

Second, assume that $A\geq 0$ and $B\geq 0$. Considering $\begin{pmatrix}
A+\frac{1}{n} & X \\
X^* & B+\frac{1}{n}
\end{pmatrix}$ and employing the first case, we get a sequence of contractions $\{K_n\}$ such that
$$(A+\frac{1}{n})^{1/2}K_n(B+\frac{1}{n})^{1/2}=X, \ \mbox{for any $n\in\mathbb{N}$}.$$ Since the closed unit ball of ${\mathbb B}({\mathscr H})$ is weakly compact, we may assume that $\{K_n\}$ converges weakly to a contraction $K\in {\mathbb B}({\mathscr H})$. Taking limits as $n$ tends to infinity, we arrive at $A^{1/2}KB^{1/2}=X$.

(ii) $\Rightarrow$ (i). Since $K$ is a contraction, by Lemma \ref{nice2} we have $\begin{pmatrix}
I & K \\
K^* & I
\end{pmatrix} \geq 0$. Hence
\begin{eqnarray*}
\begin{pmatrix}
A & X \\
X^* & B
\end{pmatrix}&=&\begin{pmatrix}
A & A^{1/2}K^*B^{1/2} \\
B^{1/2}K^*A^{1/2} & B
\end{pmatrix}\\
&=&\begin{pmatrix}
A^{1/2} & 0 \\
0 & B^{1/2}
\end{pmatrix}\begin{pmatrix}
I & K \\
K^* & I
\end{pmatrix}\begin{pmatrix}
A^{1/2} & 0 \\
0 & B^{1/2}
\end{pmatrix} \geq 0\,.
\end{eqnarray*}
\end{proof}

\begin{theorem}\label{contr2}
Let $M$ be given by \eqref{defn of M}. Then the following are equivalent:
\begin{enumerate}
\item[(i)] The block matrix $M$ is positive;
\item[(ii)] $A=Y_1^*Y_1, B=Y_2^*Y_2$, and $X=Y_1^*Y_2$ for some $Y_1, Y_2 \in {\mathbb B}( {\mathscr H},{\mathscr H}\oplus {\mathscr H})$
\end{enumerate}
\end{theorem}
\begin{proof}
\textbf{The first proof:} A proof can be carried out by following the line in the second proof of Theorem~\ref{nice1}.

\textbf{The second proof} (due to T. Ando): By Theorem \ref{contr}, there is a contraction $K$ such that $X=A^{1/2}KB^{1/2}$. It is easy to see that $G=(T_{ij})_{1\le i,j\le 2}$ is a unitary (dilation of $K$), where
$$T_{11}=K, T_{12}=(I-KK^*)^{1/2}, K_{21}=(I-K^*K)^{1/2}\ \mbox{and}\ K_{22}=-K^*.$$ Let $Y_1,Y_2\in {\mathbb B}({\mathscr H}, {\mathscr H}\oplus {\mathscr H})$ be given by
\[Y_1=G^*\begin{pmatrix}A^{1/2}\\0\end{pmatrix}\quad {\rm and}\quad Y_2=\begin{pmatrix}B^{1/2}\\0\end{pmatrix}.\]
Then direct computation yields $A=Y_1^*Y_1, B=Y_2^*Y_2$, and $X=Y_1^*Y_2$.
\end{proof}

\begin{remark}
Utilizing the polar decompositions $Y_k=U_k|Y_k|,\, k=1, 2$ with partial isometries $U_k$, we get
$X=Y_1^*Y_2=|Y_1|U_1^*U_2|Y_2|=A^{1/2}KB^{1/2}$, where $K=U_1^*U_2$ is a contraction. This gives another way to find the contraction in Theorem \ref{contr}.
\end{remark}

\begin{proposition} [A generalized Cauchy--Schwarz inequality]
Let $M$ be given by \eqref{defn of M} such that $A, B\in {\mathbb B}_{+}({\mathscr H})$.
Then $M$
is positive if and only if
\begin{eqnarray}\label{murjan}
|\langle x,Xy\rangle|^2 \leq \langle x,Ax\rangle\,\langle y,By\rangle
\end{eqnarray}
for all $x,y \in \mathscr{H}$.
\end{proposition}
\begin{proof} Suppose that
$M$
is positive. Then it follows from Theorem~\ref{contr} that $X=A^{1/2}KB^{1/2}$ for a contraction $K$. Since $K^*K\leq I$, we know that for any $x,y \in \mathscr{H}$,
\begin{align*}
|\langle x,Xy\rangle|^2=&|\langle x,A^{1/2}KB^{1/2}y\rangle|^2 =|\langle A^{1/2}x,KB^{1/2}y\rangle|^2\\
\leq& |\langle A^{1/2}x,A^{1/2}x\rangle|\, |\langle KB^{1/2}y,KB^{1/2}y\rangle|\\
&\qquad (\hbox{by the usual Cauchy-Schwartz inequality in Hilbert spaces})\\
=& |\langle x,Ax\rangle|\, |\langle y,B^{1/2}K^*KB^{1/2}y\rangle|\\
\leq&\langle x,Ax\rangle\,\langle y,By\rangle.
\end{align*}

To prove the converse suppose that \eqref{murjan} is satisfied. For any $x,y \in \mathscr{H}$,
\begin{align*}
\left\langle (x,y)^T, M(x,y)^T\right\rangle=&\langle x,Ax\rangle+2{\rm Re}\langle x,Xy\rangle+\langle y,By\rangle\\
\geq & \langle x,Ax\rangle-2|\langle x,Xy\rangle|+\langle y,By\rangle\\
\geq & \langle x,Ax\rangle-2\langle x,Ax\rangle^{\frac{1}{2}}\,\langle y,By\rangle^{\frac{1}{2}}+\langle y,By\rangle\\
=& \left(\langle x,Ax\rangle^{\frac{1}{2}}-\langle y,By\rangle^{\frac{1}{2}}\right)^2\geq 0\,.
\end{align*}
Hence $M\geq 0$.
\end{proof}
%%%%%%%%%%%%%%%%%%%%%%%%%%%%%%%%%%%%%%%%%%%%%%%%%%%%%%%%%%%%%%%%%%%%%%%%%%%%%%%%%%%%

The next result reads as follows.

\begin{proposition}\label{pr2} Let $M$ be given by \eqref{defn of M} such that $A, B, X$ are all self-adjoint and $A+B\in {\mathbb B}_{++}({\mathscr H})$. Then
$$M\geq 0\Longleftrightarrow\|(B+A)^{-1/2}(B-A+2iX)(B+A)^{-1/2}\|\leq 1.$$
\end{proposition}
\begin{proof}
Using the unitary congruent via the unitary $U=\frac{1}{\sqrt{2}}\left(\begin{array}{cc} -iI & iI \\ I & I\end{array}\right)$, one can see that $M$ is positive if and only if
$$U^*MU=\frac{1}{2}\left(\begin{array}{cc} B+A & B-A+2iX \\ B-A-2iX & B+A\end{array}\right)$$
is positive. From Theorem \ref{nice1}, the last one is positive if and only if
$$(B+A)\geq (B-A+2iX)(B+A)^{-1}(B-A-2iX)$$
and this is equivalent to $\|(B+A)^{-1/2}(B-A+2iX)(B+A)^{-1/2}\|\leq 1$.
\end{proof}
%%%%%%%%%%%%%%%%%%%%%%%%%%%%%%%%%%%%%%%%%%%%%%%%%%%%%%%%%%%%%%%%%%%%%%%%%%%

\begin{proposition}\label{pr3} Let $M$ be given by \eqref{defn of M} such that $X$ is self-adjoint and $M\ge 0$. Then
\begin{equation}\label{common upper bound of + - X}\pm X\leq \frac{1}{2\sqrt{t (1-t)}}\big[tB+(1-t)A\big], \ \mbox{for every $t\in(0,1)$}.\end{equation}
\end{proposition}
 \begin{proof}For every $t\in (0,1)$, let
 $$U_t=\left(\begin{array}{cc} (1-t)^{1/2}I & -t^{1/2}I \\ t^{1/2}I & (1-t)^{1/2}I\end{array}\right).$$ Then
$U_t$ is a unitary such that
$$U_t^*MU_t=\left(
    \begin{array}{cc}
    M_{11}(t) & M_{12}(t) \\
    M_{21}(t) & M_{22}(t) \\
    \end{array}
   \right),$$
where
\begin{align*}&M_{11}(t)= 2\sqrt{t(1-t)}X+(1-t)A+tB, \\
&M_{12}(t)=-tX+(1-t)X+\sqrt{t(1-t)}(B-A),\\
&M_{21}(t)=(1-t)X-tX+\sqrt{t(1-t)}(B-A),\\
&M_{22}(t)= -2\sqrt{t(1-t)}X+tA+(1-t)B.
\end{align*}
Then from $M_{11}(t)\ge 0$ and $M_{22}(t)\ge 0$ we can get
\begin{equation*}X\le \frac{1}{2\sqrt{t(1-t)}}\big[(1-t)A+tB\big],\ -X\le \frac{1}{2\sqrt{t(1-t)}}\big[(tA+(1-t)B\big],
\end{equation*}
which leads to
\begin{equation*}X\le \frac{1}{2\sqrt{t(1-t)}}\big[tA+(1-t)B\big],\ -X\le \frac{1}{2\sqrt{t(1-t)}}\big[(1-t)A+tB\big]
\end{equation*}
by replacing $t$ with $(1-t)$. This completes the proof of \eqref{common upper bound of + - X}.
\end{proof}

\begin{remark}Substituting $t=\frac12$ into \eqref{common upper bound of + - X} yields $\pm X\leq \frac{ A+B }{2}$, which however does not
mean the positivity of the block matrix $M$ defined by \eqref{defn of M} satisfying $A=A^*,B=B^*,X=X^*$ and $A+B\in {\mathbb B}_{++}({\mathscr H})$.
An immediate example is $X=0$, $A=2I$ and $B=-I$. However, in the converse direction we have these equivalent conditions:
\begin{align}\label{kian1}
\pm X\leq \frac{ A+B }{2} &\quad \Leftrightarrow \quad (B+A)^{-1/2}(\pm 2X)(B+A)^{-1/2} \leq I\nonumber\\
&\quad \Leftrightarrow \quad \|(2B+2A)^{-1/2}(4iX)(2B+2A)^{-1/2}\|\leq 1\nonumber\\
&\quad \Leftrightarrow \quad \left(\begin{array}{cc} A & X \\ X & B\end{array}\right)+ \left(\begin{array}{cc} B & X \\ X & A\end{array}\right)\geq 0.
\end{align}
The last equivalence follows from Proposition \ref{pr2}. But the last block matrix is $M+UMU^*$ with $U=\left(\begin{array}{cc} 0 & I\\ I & 0\end{array}\right)$. So
$$\pm X\leq \frac{ A+B }{2} \quad \Leftrightarrow \quad M+UMU^*\geq0. $$
\end{remark}
%%%%%%%%%%%%%%%%%%%%%%%%%%%%%%%%%%%%%%%%%%%%%%%%%%%%%%%%%%%%%%%%%%%%
In the special case that $A=B$, we have the following proposition:
\begin{proposition}
Let $A,X\in {\mathbb B}({\mathscr H})$ be such that $X$ is self-adjoint and $A$ is positive. Then
$$M=\left(\begin{array}{cc} A & X \\ X & A\end{array}\right)\geq0 \quad \Longleftrightarrow\quad \pm X\leq A.$$
\end{proposition}
\begin{proof}Clearly, $M\ge 0\Longleftrightarrow \left(
              \begin{array}{cc}
               A+\frac1n I& X \\
               X & A+\frac1n I \\
              \end{array}
              \right)\ge 0$ for any $n\in\mathbb{N}$, and similarly
$\pm X\leq A\Longleftrightarrow \pm X\leq A+\frac1n I$ for any $n\in\mathbb{N}$. So it suffices to prove that
for any $n\in\mathbb{N}$,
$$\left(
              \begin{array}{cc}
               A+\frac1n I& X \\
               X & A+\frac1n I \\
              \end{array}
              \right)\ge 0\Longleftrightarrow \pm X\leq A+\frac1n I.$$
This follows immediately from \eqref{kian1} since $A+\frac1n I$ is strictly positive.
\end{proof}
%%%%%%%%%%%%%%%%%%%%%%%%%%%%%%%%%%%%%%%%%%%%%%%%%%%%%%%%%%%%%%%%%%%%%%%%%%%%%%%%%%

 For example, it is evident that if $Y$ is a Hermitian (self-adjoint) matrix, then $\left(\begin{array}{cc} |Y| & Y \\ Y & |Y|\end{array}\right)\geq0$. Therefore, if $Y,Z$ are Hermitian, then
 $$ \left(\begin{array}{cc} |Y|\circ |Z| & Y\circ Z \\ Y\circ Z & |Y| \circ|Z|\end{array}\right)\geq0,$$
 where $\circ$ denotes the Hadamard product. We then obtain $ \pm Y\circ Z\leq |Y|\circ |Z|$.
 Note that the inequality $ |Y\circ Z|\leq |Y|\circ |Z|$ does not hold in general. To see this consider $Y=\left(\begin{array}{cc} 2 & 2 \\ 2 & 1\end{array}\right)$ and $Z=\left(\begin{array}{cc} 1 & 3 \\ 3 & 3\end{array}\right)$.

%%%%%%%%%%%%%%%%%%%%%%%%%%%%%%%%%%%%%%%%%%%%%%%%%%%%%%%%%%%%%%%%%%%%%%%%%%%%%%%%%%%%%%

\bibliographystyle{amsplain}

\begin{thebibliography}{99}

\bibitem{AT} W. N. Anderson, Jr. and G. E. Trapp, \textit{Shorted operators}, II. SIAM J. Appl. Math. \textbf{28} (1975), 60--71.

\bibitem{AP} A. D. Andrew and W. M. Patterson, \textit{Range inclusion and factorization of operators on classical Banach spaces}, J. Math. Anal. Appl. \textbf{156} (1991), no. 1, 40--43.

\bibitem{ANDO}T. Ando, \textit{Topics on operator inequalities}, Division of Applied Mathematics, Research Institute of Applied Electricity, Hokkaido University, Sapporo, 1978.

\bibitem{ACG1} M. L. Arias, G. Corach, and M. C. Gonzalez, \textit{Partial isometries in semi-Hilbertian spaces}, Linear Algebra Appl. \textbf{428} (2008), no. 7, 1460--1475.

\bibitem{ACG2} M. L. Arias, G. Corach, and M. C. Gonzalez, \textit{Generalized inverses and Douglas equations}, Proc. Amer. Math. Soc. \textbf{136} (2008), no. 9, 3177--3183.

\bibitem{AG} M. L. Arias and M. C. Gonzalez, \textit{Positive solutions to operator equations $AXB = C$}, Linear Algebra Appl. \textbf{433} (2010), no. 6, 1194--1202.

\bibitem{BAR} B. A. Barnes, \textit{Majorization, range inclusion, and factorization for bounded linear operators}, Proc. Amer. Math. Soc. \textbf{133} (2005), no. 1, 155--162.

\bibitem{BN} E. Beckenstein and L. Narici, \textit{Left factorization operators on topological vector spaces}, Indian J. Math. \textbf{55} (2013), 43--63.

\bibitem{BHA1} R. Bhatia, \textit{Positive definite matrices}, Princeton University Press, Princeton, 2007.

\bibitem{BOU} R. Bouldin, \textit{A counterexample in the factorization of Banach space operators}, Proc. Amer. Math. Soc. \textbf{68} (1978), no. 3, 327.

\bibitem{CHO} M. D. Choi, \textit{Some assorted inequalities for positive linear maps on $C^*$-algebras}, J. Operator Theory \textbf{4} (1980), no. 2, 271--285.

\bibitem{CF} R. Curto and L. A. Fialkow, \textit{Similarity, quasisimilarity, and operator factorizations}, Trans. Amer. Math. Soc. \textbf{314} (1989), no. 1, 225--254.

\bibitem{DOU} R. G. Douglas, \textit{On majorization, factorization, and range inclusion of operators on Hilbert space}, Proc. Amer. Math. Soc. \textbf{17} (1966), 413--415.

\bibitem{DMP} R. G. Douglas, P. S. Muhly, and C. Pearcy, \textit{Lifting commuting operators}, Michigan Math. J. \textbf{15} (1968), 385--395.

\bibitem{DK} B. P. Duggal and C. S. Kubrusly, \textit{A Putnam-Fuglede commutativity property for Hilbert space operators}, Linear Algebra Appl. \textbf{458} (2014), 108--115.

\bibitem{EMB} M. R. Embry, \textit{Factorization of operators on Banach space}, Proc. Amer. Math. Soc. \textbf{38} (1973), 587--590.

\bibitem{FYY} X. Fang, J. Yu, and H. Yao, \textit{Solutions to operator equations on Hilbert $C^*$-modules}, Linear Algebra Appl. \textbf{431} (2009), 2142--2153.

\bibitem{FMX}X. Fang, M. S. Moslehian, and Q. Xu, \textit{On majorization and range inclusion of operators on Hilbert $C^*$-modules}, Linear Multilinear Algebra \textbf{66} (2018), no. 12, 2493--2500.

\bibitem{FIA} L. A. Fialkow, \textit{Structural properties of elementary operators}, Elementary operators and applications (Blaubeuren, 1991), 55--113, World Sci. Publ., River Edge, NJ, 1992.

\bibitem{FS} L. A. Fialkow and H. Salas, \textit{Majorization, factorization and systems of linear operator equations}, Math. Balkanica (N.S.) \textbf{4} (1990), no. 1, 22--34.

\bibitem{FW} P. A. Fillmore and J. P. Williams, \textit{On operator ranges}, Advances in Math. \textbf{7} (1971), 254--281.

\bibitem{FOR} M. Forough, \textit{Majorization, range inclusion, and factorization for unbounded operators on Banach spaces}, Linear Algebra Appl. \textbf{449} (2014), 60--67.
 
\bibitem{FUJ-1}M. Fujii, \textit{Cauchy-Schwarz inequality and Riccati equation for positive semidefinite matrices}, preprint

\bibitem{FUJ-2} M. Fujii, \textit{Riccati equation for positive semidefinite matrices}, Kyoto Univ., RIMS Kokyuroku,
vol 2073, article 14, 105--113.

\bibitem{FFN}J. Fujii, M. Fujii, and R. Nakamoto, \textit{Riccati equation and positivity of operator matrices}, Kyungpook Math. J. \textbf{49} (2009), no. 4, 595--603.
 
\bibitem{HAL} P. R. Halmos, \textit{A Hilbert space problem book}, Second edition, Graduate Texts in Mathematics, 19. Encyclopedia of Mathematics and its Applications, 17. Springer-Verlag, New York-Berlin, 1982.

\bibitem{HZ} G. Hai and N. Zhang, \textit{On Fredholm completions of partial operator matrices}, Ann. Funct. Anal. \textbf{8} (2017), no. 4, 479--489

\bibitem{HJ} R. A. Horn and C. R. Johnson, \textit{Matrix analysis}, Second edition. Cambridge University Press, Cambridge, 2013. 

\bibitem{IZU} S. Izumino, \textit{Quotients of bounded operators}, Proc. Amer. Math. Soc. \textbf{106} (1989), no. 2, 427--435.

\bibitem{KHP} S. Karmakar, Sk. Hossein, and K. Paul, \textit{Properties of $J$-fusion frames in Krein spaces}, Adv. Oper. Theory \textbf{2} (2017), no. 3, 215--227.

\bibitem{LAN} E. C. Lance, \textit{Hilbert $C^*$-modules: A Toolkit for Operator Algebraists}, Cambridge University Press, Oxford, 1995.

\bibitem{LEE} R. B. Leech, \textit{Factorization of analytic functions and operator inequalities}, Integral Equations Operator Theory \textbf{78} (2014), no. 1, 71--73.

\bibitem{LDLY} C. Li, X. Duan, J. Li, and S. Yu, \textit{A new algorithm for the symmetric solution of the matrix equations $AXB=E$ and $CXD=F$}, Ann. Funct. Anal. \textbf{9} (2018), no. 1, 8--16

\bibitem{LSX} W. Luo, C. Song, and Q. Xu, The parallel sum for adjointable operators on Hilbert $C^*$-modules, preprint.

\bibitem{NOS} K. Nakade, T. Ohwada, and K.-S. Saito, \textit{Kolmogorov's factorization theorem for von Neumann algebras}, J. Math. Anal. Appl. \textbf{401} (2013), no. 1, 289--292.

\bibitem{NAK} R. Nakamoto, \textit{On the operator equation $THT = K$}, Math. Japon. \textbf{18} (1973), 251--252.

\bibitem{NAY} S. Nayak, \textit{On the diagonals of projections in matrix algebras over von Neumann algebras}, Thesis (Ph.D.), University of Pennsylvania. ProQuest LLC, Ann Arbor, MI, 2016. 74 pp.

\bibitem{MOS1} M. S. Moslehian, \textit{Trick with $2\times2$ matrices over $C^*$-algebras}, Austral. Math. Soc. Gaz. \textbf{30} (2003), no. 3, 150--157.

\bibitem{MOS2} M. S. Moslehian, \textit{On $2\times 2$ matrices over $C^*$-algebras}, Acta Math. Acad. Paedagog. Nyh\'{a}zi. (N.S.) \textbf{19} (2003), no. 1, 51--53.

\bibitem{MN} M. S. Moslehian and S. M. S. Nabavi Sales, \textit{Fuglede-Putnam type theorems via the Aluthge transform}, Positivity \textbf{17} (2013), no. 1, 151-162.

\bibitem{MUR} G. J. Murphy, \textit{$C^*$-algebras and operator theory}, Academic Press, Inc., Boston, MA, 1990.

\bibitem{NAJ} H. Najafi, \textit{Operator means and positivity of block operators}, Math. Z. \textbf{289} (2018), no. 1-2, 445--454.

\bibitem{PAU} V. I. Paulsen, \textit{Completely bounded maps and operator algebras}, Cambridge Studies in Advanced Mathematics, 78. Cambridge University Press, Cambridge, 2002.

\bibitem{PER} A. M. Peralta, \textit{Characterizing projections among positive operators in the unit sphere}, Adv. Oper. Theory \textbf{3} (2018), no. 3, 731--744.

\bibitem{PLA} J. Plastiras, \textit{$C^*$-algebras isomorphic after tensoring}, Proc. Amer. Math. Soc. \textbf{66} (1977), no. 2, 276--278.

\bibitem{PS} D. Popovici and Z. Sebestyén, \textit{Factorizations of linear relations}, Adv. Math. \textbf{233} (2013), 40--55.

\bibitem{ROD} L. Rodman, \textit{A note on indefinite Douglas' lemma}, Operator theory in inner product spaces, 225--229, Oper. Theory Adv. Appl., 175, Birkh\"auser, Basel, 2007.

\bibitem{SMU} Ju. L. \v{S}mul'jan, \textit{Division in the class of $J$-expansive operators}, (Russian) Mat. Sb. (N.S.) \textbf{74} (116) 1967 516--525.

\bibitem{VMX} M. Vosough, M. S. Moslehian, and Q. Xu, \textit{Closed range and nonclosed range adjointable operators on Hilbert $C^*$-modules}, Positivity \textbf{22} (2018), no. 3, 701--710.

\bibitem{WZ} P. Wang and X. Zhang, \textit{Range inclusion of operators on non-Archimedean Banach space}, Sci. China Math. \textbf{53} (2010), no. 12, 3215--3224.

\end{thebibliography}

\end{document}